\documentclass[reqno,10pt,centertags]{amsart}
\usepackage{amsmath,amsthm,amscd,amssymb,latexsym,enumerate}

\usepackage[final]{hyperref}

\newcommand*{\mailto}[1]{\href{mailto:#1}{\nolinkurl{#1}}}
\newcommand{\arxiv}[1]{\href{http://arxiv.org/abs/#1}{arXiv: #1}}

\makeatletter
\def\theequation{\@arabic\c@equation}

\newcommand{\bb}[1]{{\mathbb{#1}}}

\newcommand{\bbN}{{\mathbb{N}}}
\newcommand{\bbR}{{\mathbb{R}}}

\newcommand{\bbC}{{\mathbb{C}}}

\newcommand{\no}{\nonumber}
\newcommand{\lb}{\label}
\newcommand{\bi}{\bibitem}
\newcommand{\f}{\frac}

\newcommand{\la}{\lambda}

\newcommand{\Oh}{O}

\newcommand{\dom}{\operatorname{dom}}

\renewcommand{\ln}{\operatorname{ln}}
\newcommand{\tr}{\operatorname{tr}}

\numberwithin{equation}{section}

\newtheorem{theorem}{Theorem}[section]

\newtheorem{example}[theorem]{Example}
\theoremstyle{definition}

\newtheorem{remark}[theorem]{Remark}

\begin{document}
\title[Eigenvectors from Eigenvalues]{Eigenvectors From Eigenvalues: The Case of One-Dimensional Schr\"odinger Operators}

\author{Fritz Gesztesy}
\address{Department of Mathematics,
Baylor University, One Bear Place \#97328,
Waco, TX 76798-7328, USA}
\email{\mailto{Fritz\_Gesztesy@baylor.edu}}
\urladdr{\url{http://www.baylor.edu/math/index.php?id=935340}}

\author[M.\ Zinchenko]{Maxim Zinchenko}
\address{Department of Mathematics and Statistics,
University of New Mexico, Albuquerque, NM 87131, USA}
\email{\mailto{maxim@math.unm.edu}}
\urladdr{\url{http://www.math.unm.edu/~maxim/}}

\date{\today}
\thanks{Annals of Funct. Anal. (to appear).}

\begin{abstract}
We revisit an archive submission by P.\ B. Denton, S.\ J. Parke, T.\ Tao, and X.\ Zhang
\cite{DPTZ19} on $n \times n$ self-adjoint matrices from the point of view of self-adjoint Dirichlet Schr\"odinger operators on a compact interval.
\end{abstract}

\maketitle


\section{Introduction} \lb{s1}

To set the stage for this note we briefly summarize comments by Tao \cite{Ta19} on the archive submission by Denton, Parke, Tao, and Zhang \cite{DPTZ19}, and recall the following facts that illustrate an identity involving eigenvectors of a self-adjoint $n \times n$ matrix, $n \in \bbN$, its eigenvalues, and those of a closely related $(n-1) \times (n-1)$ matrix.

Let $A$ be a self-adjoint $n \times n$ matrix with eigenvalues $\lambda_1(A), \dots, \lambda_n(A)$ and corresponding normalized eigenvectors $v_1,\dots,v_n$, and denote by $v_{k,\ell}$ the $\ell$th component of $v_k$. Let $M_j$ be the $(n-1) \times (n-1)$ matrix obtained from $A$ by deleting the $j$th row and column of $A$. Let
$\lambda_1(M_j),\dots,\lambda_{n-1}(M_j)$ denote the eigenvalues of $M_j$. Then the identity proven in \cite{DPTZ19}
reads,
\begin{equation}
|v_{k,\ell}|^2 \prod_{m=1, m \neq k}^n[\lambda_k(A) - \lambda_m(A))]= \prod_{j=1}^{n-1}[\lambda_k(A)- \lambda_j(M_{\ell})].    \lb{1.1}
\end{equation}
Thus, generically, that is, as long as the product on the left-hand side in \eqref{1.1} does not vanish, the component $v_{k,\ell}$ of the eigenvector $v_k$ of $A$ is expressed in terms of the eigenvalues of $A$ and $M_{\ell}$.

A complex analytic proof of \eqref{1.1} notes that if temporarily all eigenvalues of $A$ are assumed to be simple, then
\begin{align}
& (A -z I_n)^{-1} = \sum_{k=1}^n \f{(v_k, \, \cdot \,)_{\bbC^n} v_k}{\lambda_k(A) -z},   \lb{1.2} \\
& \f{\det(M_j - zI_{n-1})}{\det(A - zI_n)}
= \sum_{k=1}^n \f{|v_{k,j}|^2}{\lambda_k(A) - z},     \lb{1.3} \\
& \f{\prod_{k=1}^{n-1} [\lambda_k(M_j) - z]}{\prod_{k=1}^n[\lambda_k(A) - z]}
= \sum_{k=1}^n \f{|v_{k,j}|^2}{\lambda_k(A) - z},     \lb{1.4}
\end{align}
hold, and taking residues in \eqref{1.4} yields \eqref{1.1}. The general case \eqref{1.1} then follows from this by continuity.

The history of identity \eqref{1.1} is rather involved as pointed out in great detail in \cite{DPTZ19} and hence we refer to the extensive list of references therein. (Here we just note that, as discussed in \cite{DPTZ19}, a related identity was already mentioned by Jacobi \cite{Ja34} in 1834 and 100 years later again by L\"owner \cite{Lo34}).

In the next section we demonstrate how to extend this circle of ideas to a class of self-adjoint second-order differential operators on a compact interval.

\section{The case of One-Dimensional Dirichlet Schr\"odinger Operators on a Compact Interval} \lb{s2}

Identity \eqref{1.1} for self-adjoint $n \times n$ matrices can be extended to one dimensional self-adjoint Dirichlet Schr\"odinger operators on compact intervals as follows.

Consider Schr\"odinger operators on a compact interval $[a,b]$ with potential $V$ satisfying
\begin{equation}
V \in L^1((a,b); dx), \, \text{$V$ real-valued a.e.~on $(a,b)$.}     \lb{2.1}
\end{equation}
The pair of matrices $(A,M_j)$ in Section \ref{s1} is then replaced by a pair of Schr\"odinger
operators $\big(H^D, H^D_{x_0}\big)$ in $L^2((a,b); dx)$, $x_0 \in (a,b)$, given by
\begin{align}
& \big(H^D f\big)(x) = - f''(x) + V(x) f(x), \quad x \in (a,b),   \no \\
& f \in \dom\big(H^D\big) = \big\{f \in L^2(a,b);dx) \, \big|\, f, f' \in AC([a,b]); \, f(a) = 0 = f(b);    \lb{2.2} \\
& \hspace*{6.4cm}  (-f'' + Vf) \in L^2(a,b);dx)\big\},   \no \\
& \big(H^D_{x_0} f\big)(x) = - f''(x) + V(x) f(x),   \quad x \in (a,x_0) \cup (x_0,b), \no \\
& f \in \dom\big(H^D_{x_0}\big) = \big\{f \in L^2(a,b);dx) \, \big|\, f, f' \in AC([a,b]); \, f(a) = 0 = f(b);    \lb{2.3} \\
& \hspace*{5cm}   \, f(x_0) = 0; \, (-f'' + Vf) \in L^2(a,b);dx)\big\},   \no
\end{align}
in particular, $H^D_{x_0}$ is thus a direct sum of two Dirichlet Schr\"odinger operators in $L^2((a,b);dx) \simeq
L^2((a,x_0);dx) \oplus L^2((x_0,b);dx)$.

Denoting by $G^D(z,x,x') = (H^D - z I)^{-1}(x,x')$, $x,x' \in [a,b]$, the Green's function (i.e., integral kernel of the resolvent) of $H^D$, and by $G^D_{x_0}(z,x,x')$ that of $H^D_{x_0}$, one verifies (the Krein-type resolvent formula, see, e.g., \cite{GS96})
\begin{align}
\begin{split}
G^D_{x_0}(z,x,x') = G^D(z,x,x') - \big[G^D(z,x_0,x_0)\big]^{-1} G^D(z,x,x_0) G^D(z,x_0,x'),&    \lb{2.4} \\
\quad z \in \bbC \backslash \big(\sigma\big(H^D_{x_0}\big) \cup \sigma(H^D)\big),&
\end{split}
\end{align}
and hence the fact
\begin{align}
\begin{split}
\tr\Big(\big(H^D_{x_0} - z I\big)^{-1} - \big(H^D - z I\big)^{-1}\Big) = - \f{d}{dz} \ln\big(G^D(z,x_0,x_0)\big),&     \lb{2.5} \\
 z \in \bbC \backslash \big(\sigma\big(H^D_{x_0}\big) \cup \sigma(H^D)\big).&
 \end{split}
\end{align}
Here we employed $(d/dz) (H^D-zI)^{-1} = (H^D-zI)^{-2}$ and the fact that the second term on the right-hand side of \eqref{2.4} represents the integral kernel of a rank-one operator. To compute $\tr((H^D_{x_0} - z I)^{-1})$ and
$\tr((H^D - z I)^{-1})$ as the integral of the diagonal Green's functions (i.e., integral kernels on the diagonal) one first assumes that $z < 0$ is sufficiently negative to render $(H^D_{x_0} - z I)^{-1}$ and $(H^D - z I)^{-1}$ nonnegative operators in $L^2((a,b);dx)$ and applies
\cite[p.~65--66]{RS79}, or Mercer's theorem as in \cite[Proposition~5.6.9]{Da07}, followed by analytic continuation with respect to $z \in \bbC \backslash (\sigma(H^D_{x_0}) \cup \sigma(H^D))$.

A comparison of \eqref{2.5} with the standard trace formula
\begin{equation}
\tr\big((B - z I)^{-1} - (B_0 - z I)^{-1}\big) = - \f{d}{dz} \ln(D_{B_0,B}(z)),
\end{equation}
where $B-B_0$ is trace class and $D_{B_0,B}(\, \cdot \,)$ denotes the (relative) Fredholm determinant
\begin{equation}
D_{B_0,B}(z) = \det\big((B-zI)(B_0-zI)^{-1}\big) = \det\big(I + (B-B_0)(B_0-zI)^{-1}\big)
\end{equation}
for the pair of operators $(B_0,B)$ (this extends to much more general situations where only
$[(B - z I)^{-1} - (B_0 - z I)^{-1}]$ is trace class), shows that $G^D(z,x_0,x_0)$ is the precise analog of the Fredholm determinant for the pair $(H^D,H^D_{x_0})$. In particular, $G^D(z,x_0,x_0)$ is the Schr\"odinger operator analog of
\begin{equation}
{\det}_{\bbC^{n-1}}(M_j - zI_{n-1})/{\det}_{\bbC^n}(A - zI_n)
\end{equation}
in \eqref{1.3}.

Denoting
\begin{equation}
\sigma\big(H^D\big) = \big\{\lambda_k\big(H^D\big)\big\}_{k \in \bbN},     \lb{2.9}
\end{equation}
(all eigenvalues of $H^D$ are necessarily simple in this self-adjoint context) the analytic structure of $G^D(z,x,x')$ is of the form
\begin{equation}
G^D(z,x,x') = \sum_{k \in \bbN} \big(\lambda_k\big(H^D\big) - z\big)^{-1} e_k(x) e_k(x'), \quad
z \in \bbC \big\backslash\sigma\big(H^D\big), \;  x, x' \in [a,b],    \lb{2.10}
\end{equation}
with $\{e_k(x)\}_{k \in \bbN} \subset L^2((a,b);dx)$ the normalized eigenfunctions of $H^D$ associated with the
eigenvalues $\lambda_k(H^D)$ of $H^D$. Moreover, the Green's function of any such Schr\"odinger operator is of the form
\begin{equation}
G^D(z,x,x') = W(\psi_b(z,\,\cdot\,), \psi_a(z,\,\cdot\,))^{-1} \begin{cases} \psi_a(z,x) \psi_b(z,x'), & a \leq x \leq x' \leq b, \\
\psi_a(z,x') \psi_b(z,x), & a \leq x' \leq x \leq b,
\end{cases}
\end{equation}
where $\psi_a(z,x)$ (resp., $\psi_b(z,x)$) are distributional solutions of
$H^D \psi(z, \, \cdot \,) = z \psi(z,\, \cdot \,)$, entire in $z$ for fixed $x \in [a,b]$, satisfying the Dirichlet boundary condition at $a$, respectively, at $b$, that is,
\begin{equation}
\psi_a(z,a) = 0, \quad \psi_b(z,b) = 0, \quad z \in \bbC.
\end{equation}
Here $W(f,g) = f(x)g'(x) - f'(x) g(x)$, $x \in [a,b]$, denotes the Wronskian of $f$ and $g$ (with $f,g$ absolutely continuous
on $[a,b]$).

Denoting
\begin{equation}
\sigma\big(H^D_{x_0}\big) = \big\{\lambda_k\big(H^D_{x_0}\big)\big\}_{k \in \bbN},
\lb{2.12a}
\end{equation}
and abbreviating the multiplicity of $\lambda_k\big(H^D_{x_0}\big)$ by $m_k^D(x_0) \in \{1,2\}$,
a first analog of \eqref{1.1} then reads as follows.

\begin{theorem} \lb{t2.1}
Assume \eqref{2.1} and introduce $H^D$, $H_{x_0}^D$, $\sigma\big(H^D\big)$, $\sigma\big(H_{x_0}^D\big)$ as in \eqref{2.1}--\eqref{2.3}, \eqref{2.9} and suppose that
$0 \notin \sigma\big(H^D\big) \cap\sigma\big(H_{x_0}^D\big)$. Then,
\begin{align}
\begin{split}
e_k(x_0)^2 &= \left[\lim_{z \downarrow - \infty} |z|^{-1/2} \,
\f{\displaystyle{\prod_{n\in\bbN} \bigg(1-\f{z}{\lambda_n\big(H^D\big)}\bigg)}}
{\displaystyle{\prod_{\ell \in \bbN} \bigg(1-\f{z}{\lambda_{\ell}\big(H^D_{x_0}\big)}\bigg)^{m_{\ell}^D(x_0)}}}\right]
\\
& \quad \times \lambda_k\big(H^D\big)
\f{  \displaystyle{\prod_{j \in \bbN} \bigg(1 - \f{\lambda_k\big(H^D\big)}{\lambda_j\big(H_{x_0}\big)}\bigg)^{m_j^D(x_0)}}}{\displaystyle{\prod_{m\in\bbN\backslash\{k\}} \bigg(1 - \f{\lambda_k\big(H^D\big)}{\lambda_m\big(H^D\big)}\bigg)}}.   \lb{2.13}
\end{split}
\end{align}
\end{theorem}
\begin{proof}
Since for fixed $x \in [a,b]$, $\psi_a(z,x), \psi_b(z,x)$ and $W(\psi_a(z,\,\cdot\,), \psi_b(z,\,\cdot\,))(x)$ are entire functions with respect to $z$ of order $1/2$ (cf.\ \cite{Le49}), $G^D(z,x_0,x_0)$ has a product representation of the form
\begin{align}
G^D(z,x_0,x_0) &= C(x_0) \f{\displaystyle{\prod_{\ell \in \bbN} \bigg(1 - \f{z}{\lambda_{\ell}\big(H^D_{x_0}\big)}\bigg)^{m_{\ell}^D(x_0)}}}{\displaystyle{\prod_{k\in\bbN} \bigg(1 - \f{z}{\lambda_k\big(H^D\big)}\bigg)}},     \lb{2.14}
\\
C(x_0) &= G^D(0,x_0,x_0),
\end{align}
as $G^D(z,x_0,x_0)$ is meromorphic in $z$ with simple poles at $\sigma(H^D)$ and at most double zeros at
$\sigma(H^D_{x_0})$. Thus, \eqref{2.10} and \eqref{2.14} imply
\begin{align}
e_k(x_0)^2 &= \lim_{z \to \lambda_k(H^D)} [\lambda_k\big(H^D\big) -z] G^D(z,x_0,x_0) \no \\
&= C(x_0)  \lim_{z \to \lambda_k(H^D)} [\lambda_k\big(H^D\big) -z]
\f{\displaystyle{\prod_{j \in \bbN} \bigg(1 - \f{z}{\lambda_j\big(H^D_{x_0}\big)}\bigg)^{m_j^D(x_0)}}}{\displaystyle{\prod_{m\in\bbN} \bigg(1 - \f{z}{\lambda_m\big(H^D\big)}\bigg)}}   \no \\
&= C(x_0) \lim_{z \to \lambda_k(H^D)} [\lambda_k\big(H^D\big) -z]
\f{\displaystyle{\prod_{j \in \bbN} \bigg(1 - \f{z}{\lambda_j\big(H^D_{x_0}\big)}\bigg)^{m_j^D(x_0)}}}{\displaystyle{\prod_{m\in\bbN\backslash\{k\}} \bigg(1 - \f{z}{\lambda_m\big(H^D\big)}\bigg)
\bigg[1 - \f{z}{\lambda_k\big(H^D\big)}\bigg]}}   \no \\
&= C(x_0) \lambda_k\big(H^D\big)
\f{  \displaystyle{\prod_{j \in \bbN} \bigg(1 - \f{\lambda_k\big(H^D\big)}{\lambda_j\big(H^D_{x_0}\big)}\bigg)^{m_j^D(x_0)}}}{\displaystyle{\prod_{m\in\bbN\backslash\{k\}} \bigg(1 - \f{\lambda_k\big(H^D\big)}{\lambda_m\big(H^D\big)}\bigg)}}.                   \lb{2.17}
\end{align}

To determine the normalization constant $C(x_0)$ from the two sets of eigenvalues \eqref{2.9} one can use the representation \eqref{2.14} and the asymptotics (see, \cite[Eqs.~(3.28), (A.40)]{GS96} and the references cited therein),
\begin{equation}
\lim_{z \downarrow - \infty} |z|^{1/2} \, G^D(z,x_0,x_0) = 1,    \lb{2.15}
\end{equation}
which yield
\begin{align}
C(x_0) = \lim_{z \downarrow - \infty} |z|^{-1/2} \,
\f{\displaystyle{\prod_{n\in\bbN} \bigg(1-\f{z}{\lambda_n\big(H^D\big)}\bigg)}}
{\displaystyle{\prod_{\ell \in \bbN} \bigg(1-\f{z}{\lambda_{\ell}\big(H^D_{x_0}\big)}\bigg)^{m_{\ell}^D(x_0)}}}.   \lb{2.18}
\end{align}
\end{proof}

\begin{remark} \lb{r2.2}
If $0 \in \sigma\big(H^D\big) \cup\sigma\big(H_{x_0}^D\big)$, one can always shift $V$ by a sufficiently small additive constant so that $0$ does not belong to $\sigma\big(H^D\big) \cup \sigma\big(H_{x_0}^D\big)$. Alternatively, one can derive the analog of \eqref{2.13} directly as follows: \\[1mm]
$(i)$ Suppose $\lambda_{k_0}(H^D) = 0$ and $0 \notin \sigma\big(H_{x_0}^D\big)$. Then the same considerations as in the proof of Theorem \ref{t2.1} yield
\begin{align}
\begin{split}
e_{k_0}(x_0)^2 &= \lim_{z \downarrow - \infty} |z|^{1/2} \,
\f{\displaystyle{\prod_{n\in\bbN \backslash \{k_0\}} \bigg(1-\f{z}{\lambda_n\big(H^D\big)}\bigg)}}
{\displaystyle{\prod_{\ell \in \bbN} \bigg(1-\f{z}{\lambda_{\ell}\big(H^D_{x_0}\big)}\bigg)^{m_{\ell}^D(x_0)}}},
\\
e_{k}(x_0)^2 &= - \left[\lim_{z \downarrow - \infty} |z|^{1/2} \,
\f{\displaystyle{\prod_{n\in\bbN \backslash \{k_0\}} \bigg(1-\f{z}{\lambda_n\big(H^D\big)}\bigg)}}
{\displaystyle{\prod_{\ell \in \bbN} \bigg(1-\f{z}{\lambda_{\ell}\big(H^D_{x_0}\big)}\bigg)^{m_{\ell}^D(x_0)}}}\right] \\
& \qquad \times \lambda_k\big(H^D\big)^{-1} \,
\f{  \displaystyle{\prod_{j \in \bbN} \bigg(1 - \f{\lambda_k\big(H^D\big)}{\lambda_j\big(H_{x_0}\big)}\bigg)^{m_j^D(x_0)}}}{\displaystyle{\prod_{m\in\bbN\backslash\{k_0\}} \bigg(1 - \f{\lambda_k\big(H^D\big)}{\lambda_m\big(H^D\big)}\bigg)}}, \quad k \in \bbN \backslash \{k_0\}.
\end{split}
\end{align}
$(ii)$ Suppose $\lambda_{k_0}(H^D) = 0$ and $\lambda_{\ell_0}\big(H_{x_0}^D\big) = 0$. Then $0$ is a twice degenerate eigenvalue of $H_{x_0}^D$ and
\begin{align}
\begin{split}
e_{k_0}(x_0) &= 0,   \\
e_k(x_0)^2 &= - \left[\lim_{z \downarrow - \infty} |z|^{-3/2} \,
\f{\displaystyle{\prod_{n\in\bbN \backslash \{k_0\}} \bigg(1-\f{z}{\lambda_n\big(H^D\big)}\bigg)}}
{\displaystyle{\prod_{\ell \in \bbN \backslash \{\ell_0\}} \bigg(1-\f{z}{\lambda_{\ell}\big(H^D_{x_0}\big)}\bigg)^{m_{\ell}^D(x_0)}}}\right] \\
& \qquad \times \lambda_k\big(H^D\big)^2 \,
\f{  \displaystyle{\prod_{j \in \bbN \backslash \{\ell_0\}} \bigg(1 - \f{\lambda_k\big(H^D\big)}{\lambda_j\big(H_{x_0}\big)}\bigg)^{m_j^D(x_0)}}}{\displaystyle{\prod_{m\in\bbN\backslash\{k_0,k\}} \bigg(1 - \f{\lambda_k\big(H^D\big)}{\lambda_m\big(H^D\big)}\bigg)}}, \quad k \in \bbN \backslash \{k_0\}.
\end{split}
\end{align}
$(iii)$ Suppose $0 \notin \sigma\big(H^D\big)$ and $\lambda_{\ell_0}\big(H^D_{x_0}\big) = 0$. Then $0$ is a simple eigenvalue of $H_{x_0}^D$ and
\begin{align}
\begin{split}
e_k(x_0)^2 &= - \left[\lim_{z \downarrow - \infty} |z|^{-3/2} \,
\f{\displaystyle{\prod_{n\in\bbN} \bigg(1-\f{z}{\lambda_n\big(H^D\big)}\bigg)}}
{\displaystyle{\prod_{\ell \in \bbN \backslash \{\ell_0\}} \bigg(1-\f{z}{\lambda_{\ell}\big(H^D_{x_0}\big)}\bigg)^{m_{\ell}^D(x_0)}}}\right] \\
& \qquad \times \lambda_k\big(H^D\big)^2 \,
\f{  \displaystyle{\prod_{j \in \bbN \backslash \{\ell_0\}} \bigg(1 - \f{\lambda_k\big(H^D\big)}{\lambda_j\big(H_{x_0}\big)}\bigg)^{m_j^D(x_0)}}}{\displaystyle{\prod_{m\in\bbN\backslash\{k\}} \bigg(1 - \f{\lambda_k\big(H^D\big)}{\lambda_m\big(H^D\big)}\bigg)}}, \quad k \in \bbN.
\end{split}
\end{align}
${}$ \hfill $\diamond$
\end{remark}

Before continuing this discussion we consider an exactly solvable example next.

\begin{example} \lb{e2.3}
Let $V=0$~a.e.~on $(a,b)$, and introduce in $L^2((a,b); dx)$,
\begin{align}
& \big(H_0^D f\big)(x) = - f''(x), \quad x \in (a,b),   \no \\
& f \in \dom\big(H_0^D\big) = \big\{f \in L^2(a,b);dx) \, \big|\, f, f' \in AC([a,b]); \, f(a) = 0 = f(b);    \lb{2.19} \\
& \hspace*{8.3cm}  f'' \in L^2(a,b);dx)\big\}   \no \\
& \hspace*{2cm} = H^2((a,b)) \cap H^1_0((a,b)),   \no \\
& \big(H^D_{0,x_0} f\big)(x) = - f''(x),  \quad x \in (a,x_0) \cup (x_0,b), \no \\
& f \in \dom\big(H^D_{0,x_0}\big) = \big\{f \in L^2(a,b);dx) \, \big|\, f, f' \in AC([a,b]); \, f(a) = 0 = f(b);    \lb{2.20} \\
& \hspace*{6.75cm}   \, f(x_0) = 0; \, f'' \in L^2(a,b);dx)\big\}.   \no
\end{align}
Then one obtains,
\begin{align}
&G_0^D(z;x,x') = \frac{1}{z^{1/2} \sin\big(z^{1/2}(b-a)\big)}    \no \\
& \hspace*{2.15cm} \times \begin{cases}  \sin\big(z^{1/2} (x-a)\big) \sin\big(z^{1/2} (b-x')\big), & a \leq x \leq x' \leq b,
\\[1mm]
\sin \big(z^{1/2} (x'-a)\big) \sin\big(z^{1/2} (b-x)\big), & a \leq x' \leq x \leq b,
\end{cases}   \\
& C_0(x) = G_0^D(0,x,x) = (x-a)(b-x)/(b-a),     \lb{2.22} \\
& \sigma\big(H_0^D\big) = \big\{[k\pi/(b-a)]^2\big\}_{k \in \bbN},    \lb{2.23} \\
& e_{0,k}(x) = [2/(b-a)]^{1/2} \sin(k \pi (x-a)/(b-a)), \quad a \leq x \leq b,    \lb{2.24} \\
& \sigma\big(H^D_{0,x_0}\big) = \big\{[j\pi\big/(x_0-a)]^2\big\}_{j \in \bbN}
\cup  \big\{[\ell \pi\big/(b-x_0)]^2\big\}_{\ell \in \bbN}, \quad x_0 \in (a,b).   \lb{2.25}
\end{align}
Applying \eqref{2.17} one computes
\begin{align}
e_{0,k}(x_0)^2 &= \f{(x_0-a)(b-x_0)}{(b-a)} \f{(k \pi)^2}{(b-a)^2}
\displaystyle{\prod_{j \in \bbN} \bigg(1 - \f{k^2 (x_0-a)^2}{j^2 (b-a)^2}\bigg)}    \no \\
& \quad \times \displaystyle{\prod_{\ell \in \bbN} \bigg(1 - \f{k^2 (b-x_0)^2}{\ell^2 (b-a)^2}\bigg)} \bigg/
\displaystyle{\prod_{m \in \bbN\backslash \{k\}}  \bigg(1 - \f{k^2}{m^2}\bigg)}    \no \\
&= \f{(x_0-a)(b-x_0)}{(b-a)} \f{(k \pi)^2}{(b-a)^2}    \no \\
& \quad \times \f{\sin(k \pi (x_0-a)/(b-a))}{k \pi (x_0-a)/(b-a)}
\f{\sin(k \pi (b-x_0)/(b-a))}{k \pi (b-x_0)/(b-a)}   \no \\
& \quad \times \bigg[\prod_{m \in \bbN\backslash \{k\}}  \bigg(1 - \f{k^2}{m^2}\bigg)\bigg]^{-1}   \no \\
& = [2/(b-a)] [\sin(k \pi (x_0-a)/(b-a))]^2, \quad x_0 \in [a,b],     \lb{2.26}
\end{align}
confirming \eqref{2.24}.

Here we employed $\sin(\zeta) = \zeta \prod_{j \in \bbN} \big[1 - \zeta^2j^{-2} \pi^{-2}\big]$, $\zeta \in \bbC$, which also yields the following identity, required to arrive at \eqref{2.26},
\begin{equation}
\prod_{m \in \bbN\backslash \{k\}}  \bigg(1 - \f{k^2}{m^2}\bigg) = (-1)^{k+1}/2, \quad k \in \bbN.
\end{equation}
\end{example}

Next, we employ Example \ref{e2.3} to elaborate on the computation of $C(x_0)$ in \eqref{2.18}. In this manner the squared eigenvectors $e_k(x_0)^2$ can be expressed in terms of the two sets of eigenvalues \eqref{2.9} somewhat more explicitly.

\begin{theorem} \lb{t2.4}
Assume \eqref{2.1} and introduce the linear operators $H^D$, $H_{x_0}^D$, $H_0^D$, $H_{0,x_0}^D$,
and the spectra
$\sigma\big(H^D\big)$, $\sigma\big(H_{x_0}^D\big)$, $\sigma\big(H_0^D\big)$, $\sigma\big(H_{0,x_0}^D\big)$ as in \eqref{2.1}--\eqref{2.3}, \eqref{2.9}, \eqref{2.12a}, \eqref{2.19}, \eqref{2.20}, \eqref{2.23}, and \eqref{2.25}.
Then,
\begin{align}
e_k(x)^2 = \f{(x_0-a)(b-x_0)}{(b-a)}
\lambda_k\big(H_0^D\big)
\f{\displaystyle{\prod_{j\in\bbN}
\bigg(\f{\lambda_j\big(H^D_{x_0}\big)-\lambda_k\big(H^D\big)}{\lambda_j\big(H^D_{0,x_0}\big)}\bigg)}}
{\displaystyle{\prod_{m\in\bbN\backslash\{k\}} \bigg(\f{\lambda_m\big(H^D\big)
- \lambda_k\big(H^D\big)}{\lambda_m\big(H_0^D\big)}\bigg)}}.
\lb{2.28}
\end{align}
\end{theorem}
\begin{proof}
Recalling \eqref{2.23} and \eqref{2.25}, $0 \notin \sigma\big(H_0^D\big) \cap \sigma\big(H^D_{0,x_0}\big)$, and
$\la_k(H_0^D)=[k\pi/(b-a)]^2$, $k \in \bbN$. Thus, employing \cite[Eq.~(1.6.6)]{LG64}, one has
\begin{align}
\la_k\big(H^D\big)^{1/2}  \underset{k \to \infty}{=} \f{k\pi}{b-a} + \Oh\bigg(\f{1}{k}\bigg)
\end{align}
and hence
\begin{align}
\f{\la_k\big(H^D\big)}{\la_k\big(H_0^D\big)} \underset{k \to \infty}{=} 1+\Oh\bigg(\f{1}{k^2}\bigg).
\end{align}
Similarly, labeling the eigenvalues of $H^D_{0,x_0}$ and $H^D_{x_0}$ in a convenient manner so that the even ones correspond to the problem on $(a,x_0)$ and the odd ones to the problem on $(x_0,b)$ (e.g.,
$\la_{2k}(H^D_{0,x_0})=[k\pi/(x_0-a)]^2$, $\la_{2k-1}(H^D_{0,x_0})=[k\pi/(b-x_0)]^2$, $k\in\bbN$\big), one gets
\begin{align}
    \f{\la_k\big(H^D_{x_0}\big)}{\la_k\big(H^D_{0,x_0}\big)}  \underset{k \to \infty}{=} 1+\Oh\Big(\f{1}{k^2}\Big).
\end{align}
Then it follows that the infinite products
\begin{equation}
\prod_{k\in\bbN} [\la_k\big(H^D\big)/\la_k\big(H_0^D\big)] \, \text{ and } \,
\prod_{k\in\bbN} \big[\la_k\big(H^D_{x_0}\big)\big/\la_k\big(H^D_{0,x_0}\big)\big]
\end{equation}
converge to finite nonzero values and moreover one has
\begin{align}
\begin{split}
&\lim_{n\to\infty}\prod_{k=n}^\infty \min\left\{\f{\la_k\big(H^D\big)}{\la_k\big(H_0^D\big)},1\right\}=1,
\quad
\lim_{n\to\infty}\prod_{k=n}^\infty \max\left\{\f{\la_k\big(H^D\big)}{\la_k\big(H_0^D\big)},1\right\}=1,    \\
&\lim_{n\to\infty}\prod_{k=n}^\infty \min\left\{\f{\la_k\big(H^D_{x_0}\big)}{\la_k\big(H^D_{0,x_0}\big)},1\right\}=1,
\quad
\lim_{n\to\infty}\prod_{k=n}^\infty \max\left\{\f{\la_k\big(H^D_{x_0}\big)}{\la_k\big(H^D_{0,x_0}\big)},1\right\}=1.
\lb{2.33}
\end{split}
\end{align}
Next, consider the Green's function $G_0^D(z,x,x')$ for the operator $H_0^D$. Multiplying and dividing by
$G_0^D(z,x_0,x_0)$ under the limit in \eqref{2.18} and utilizing \eqref{2.14} and \eqref{2.15} for $G_0^D(z,x_0,x_0)$
one obtains (cf.\ \eqref{2.22})
\begin{align}
C(x_0) &= \lim_{z \downarrow - \infty} G_0^D(z,x_0,x_0) \,
\f{\displaystyle{\prod_{k\in\bbN} \bigg(1-\f{z}{\lambda_k\big(H^D\big)}\bigg)}}
{\displaystyle{\prod_{\ell \in \bbN} \bigg(1-\f{z}{\lambda_{\ell}\big(H^D_{x_0}\big)}\bigg)}}    \no
\\
&= C_0(x_0)\lim_{z \downarrow - \infty}
\f{\displaystyle{\prod_{\ell \in \bbN} \bigg(1-\f{z}{\lambda_{\ell}\big(H^D_{0,x_0}\big)}\bigg)}}
{\displaystyle{\prod_{k\in\bbN} \bigg(1-\f{z}{\lambda_k\big(H_0^D\big)}\bigg)}}
\f{\displaystyle{\prod_{k\in\bbN} \bigg(1-\f{z}{\lambda_k\big(H^D\big)}\bigg)}}
{\displaystyle{\prod_{\ell \in \bbN} \bigg(1-\f{z}{\lambda_{\ell}\big(H^D_{x_0}\big)}\bigg)}}    \no
\\
&= C_0(x_0)\lim_{z \downarrow - \infty}
{\displaystyle\prod_{\ell\in\bbN} \f{1-z/\lambda_{\ell}\big(H^D_{0,x_0}\big)}{1-z/\lambda_{\ell}\big(H^D_{x_0}\big)}}\bigg/
{\displaystyle\prod_{k\in\bbN}
\f{1-z/\lambda_k\big(H_0^D\big)}{1-z/\lambda_k\big(H^D\big)}}.              \lb{2.34}
\end{align}
Without loss of generality we assumed implicitly in arriving at \eqref{2.34} that
$0 \notin \sigma\big(H^D\big) \cap \sigma\big(H^D_{x_0}\big)$. (Again, one could temporarily shift $V$ by a sufficiently small additive constant and note that \eqref{2.28} only depends on differences of eigenvalues of $H^D$, resp., $H^D_{x_0}$.)

Next, we note that for any $n\in\bb N$,
\begin{align}
\lim_{z \downarrow - \infty}
{\displaystyle\prod_{\ell=1}^n \f{1-z/\lambda_{\ell}\big(H^D_{0,x_0}\big)}{1-z/\lambda_{\ell}\big(H^D_{x_0}\big)}}\bigg/
{\displaystyle\prod_{k=1}^n
\f{1-z/\lambda_k\big(H_0^D\big)}{1-z/\lambda_k\big(H^D\big)}} =
{\displaystyle\prod_{\ell=1}^n \f{\lambda_{\ell}\big(H^D_{x_0}\big)}{\lambda_{\ell}\big(H^D_{0,x_0}\big)}}\bigg/
{\displaystyle\prod_{k=1}^n
\f{\lambda_k\big(H^D\big)}{\lambda_k\big(H_0^D\big)}}.                      \lb{2.35}
\end{align}
For all sufficiently large $\ell,k$ the eigenvalues $\lambda_{\ell}(H^D_{x_0}), \lambda_k(H^D),
\lambda_{\ell}(H^D_{0,x_0}), \lambda_k(H_0^D)$ are positive and hence for all $z<0$ we have the estimates
\begin{align}
\min\left\{
\f{\lambda_{\ell}\big(H^D_{x_0}\big)}{\lambda_{\ell}\big(H^D_{0,x_0}\big)},1
\right\}
&\leq
\left|\f{1-z/\lambda_{\ell}\big(H^D_{0,x_0}\big)} {1-z/\lambda_{\ell}\big(H^D_{x_0}\big)}\right|
\leq
\max\left\{
\f{\lambda_{\ell}\big(H^D_{x_0}\big)}{\lambda_{\ell}\big(H^D_{0,x_0}\big)},1
\right\},
\\
\min\left\{\f{\lambda_k\big(H^D\big)}{\lambda_k\big(H_0^D\big)},1\right\}
&\leq
\left|\f{1-z/\lambda_k\big(H_0^D\big)}{1-z/\lambda_k\big(H^D\big)}\right|
\leq
\max\left\{\f{\lambda_k\big(H^D\big)}{\lambda_k\big(H_0^D\big)},1\right\}.
\end{align}
Combining these estimates with \eqref{2.33} and \eqref{2.35} then shows that the limit in \eqref{2.34} can be evaluated term by term. In addition, noting that $C_0(x)=G_0^D(0,x,x)=(x-a)(b-x)/(b-a)$ by \eqref{2.22}, one obtains
\begin{align}
C(x_0) &= \f{(x_0-a)(b-x_0)}{(b-a)}
{\displaystyle\prod_{\ell\in\bbN} \f{\lambda_{\ell}\big(H^D_{x_0}\big)}{\lambda_{\ell}\big(H^D_{0,x_0}\big)}}\bigg/
{\displaystyle\prod_{k\in\bbN}
\f{\lambda_k\big(H^D\big)}{\lambda_k\big(H_0^D\big)}}.                          \lb{2.38}
\end{align}
Finally, combining \eqref{2.38} with \eqref{2.17} yields \eqref{2.28}.
\end{proof}

\begin{remark} \lb{r2.5}
$(i)$ This is just the tip of the iceberg as more general situations can be discussed along similar lines: Three-coefficient  Sturm--Liouville operators in $L^2((a,b); r(x)dx)$ generated by differential expressions of the type
\begin{equation}
r(x)^{-1} [(d/dx) p(x) (d/dx) + q(x)];
\end{equation}
non-self-adjoint operators as long as the analog of $\lambda_k(H^D)$ has algebraic multiplicity equal to one (i.e., the associated Riesz projection is one-dimensional); other one-dimensional systems such as Jacobi operators, CMV operators, etc.; other boundary conditions (Neumann, Robin, etc.); in principle, this circle of ideas extends to some higher-order one-dimensional systems. \\[1mm]
$(ii)$ The standard results on ``two spectra determine the potential uniquely" (see \cite{Bo46}, \cite{Bo52}, \cite{Le49}) and the corresponding reconstruction results of the potential (cf.\ \cite{Le68}, \cite[Ch.~3]{Le87}, \cite{LG64}, \cite[Sects.~6.9, 6.11]{LS91}, \cite{Ma73}, \cite[Sect.~3.4]{Ma11}) differ from the results discussed in this note as the ``two-spectra results'' vary the boundary condition on one end, but keep the one at the opposite endpoint fixed. However, the ``three spectra results'' in \cite{BPY16}, \cite{GS99}, \cite{Pi99} are related as they include as a special case the situation of Dirichlet boundary conditions on the intervals $(a,b)$, $(a,x_0)$ and $(x_0,b)$ associated with the Dirichlet Schr\"odinger operators $H^D$ and $H_{(a,x_0)}^D$, $H_{(x_0,b)}^D$, where $H_{x_0}^D = H_{(a,x_0)}^D \oplus H_{(x_0,b)}^D$, and $H_{(a,x_0)}^D$ and $H_{(x_0,b)}^D$ represent the Dirichlet Schr\"odinger operators associated with the intervals $(a,x_0)$ and $(x_0,b)$, respectively. The result in \cite{BPY16}, \cite{GS99}, \cite{Pi99} most relevant to our situation then reads as follows: \\[1mm]
{\it Suppose that the three sets $\sigma(H^D)$, $\sigma(H_{(a,x_0)}^D)$, and $\sigma(H_{(x_0,b)}^D)$ are mutually disjoint. Then $\sigma(H^D)$, $\sigma(H_{(a,x_0)}^D)$, and $\sigma(H_{(x_0,b)}^D)$ determine
$V(\,\cdot\,)$ uniquely a.e.~on $[a,b]$.} (It can be shown that there are counterexamples to this statement if the disjointness condition of the three spectra is violated, see \cite{GS99}). \\[1mm]
$(iii)$ The idea of using spectra associated with the pair of operators $(H^D,H^D_{x_0})$ is closely related to the concept of Krein's spectral shift function $\xi (\, \cdot \,; H^D,H^D_{x_0})$ for this pair. The latter function is  explicitly given by
\begin{equation}
\xi \big(\lambda; H^D,H^D_{x_0}\big) = \pi^{-1}
\arg\big(\lim_{\varepsilon \downarrow 0} \big(G^D(\lambda + i \varepsilon, x,x\big)\big),
\quad x \in (a,b), \; \text{for a.e.~$\lambda \in \bbR$},
\end{equation}
such that \eqref{2.5} can be complemented by
\begin{align}
\tr\Big(\big(H^D_{x_0} - z I\big)^{-1} - \big(H^D - z I\big)^{-1}\Big) &= - \f{d}{dz} \ln\big(G^D(z,x_0,x_0)\big),
\no \\
&= - \int_{E_0}^{\infty} \f{\xi \big(\, \cdot \,; H^D,H^D_{x_0}\big) \, d\lambda}{(\lambda - z)^2},   \\
& \hspace*{-2.8cm} E_0 = \inf\big(\sigma\big(H^D\big)\big), \quad z \in \bbC \backslash \big(\sigma\big(H^D_{x_0}\big) \cup \sigma(H)\big).   \no
\end{align}
In the compact interval case at hand, $\xi (\, \cdot \,; H^D,H^D_{x_0})$ is a piecewise constant function, normalized to be $0$ for
$\lambda < E_0$, jumping by $1$ (resp., $-1$) at each eigenvalue of $H^D$ (resp., $H^D_{x_0}$). One can now use the exponential Herglotz--Nevanlinna representation of $\ln\big(G^D (z,x_0,x_0)\big)$ in terms of the
measure $\xi \big(\, \cdot \,; H^D,H^D_{x_0}\big) \, d\lambda$ to obtain an alternative derivation of
$G^D (z,x_0,x_0)$ as in \eqref{2.14}, \eqref{2.18}. In the related situation where $(a,b) = \bbR$ and $V$ is continuous and bounded from below, this circle of ideas has been used in \cite{GS96a} to derive a trace formula for the potential $V(x)$ in terms of $\xi (\, \cdot \,; H^D,H^D_{x})$, $x \in \bbR$; moreover, a variety of inverse spectral questions associated with half-line and real line problems involving Krein's spectral shift function were discussed in \cite{GS96}. \hfill $\diamond$
\end{remark}

\noindent {\bf Acknowledgments.} F.G.\ is indebted to Yuri Tomilov for kindly pointing out to him the archive submission  by P.\ B.\ Denton, S.\ J.\ Parke, T.\ Tao, and X.\ Zhang \cite{DPTZ19}. M.Z.\ is grateful for the hospitality of the Mathematics Department of Baylor University, where much of this work was done. Moreover, we are indebted to the anonymous referee for a very careful reading of our manuscript.


\end{document}